\documentclass [11pt]{article}
\usepackage{amsmath,amsfonts}
\usepackage{amssymb}
\usepackage{amsmath}
\usepackage{amsthm}
\usepackage{algorithmic}
\usepackage{algorithm}
\usepackage{array}
\usepackage[caption=false,font=normalsize,labelfont=bf,textfont=it]{subfig}
\usepackage{textcomp}
\usepackage{stfloats}
\usepackage{url}
\usepackage{verbatim}
\usepackage{graphicx}
\usepackage{cite}
\usepackage[left=2.3cm,right=2.3cm,top=2.85cm,bottom=2.55cm]{geometry}
\newtheorem{thm}{\bf Theorem}[section]
\newtheorem{defn}[thm]{\bf Definition}

\newtheorem{lem}[thm]{\bf Lemma}
\newtheorem{cor}[thm]{\bf Corollary}
\newtheorem{rem}[thm]{\bf Remark}

\newtheorem{example}[thm]{\bf{Example}}

\newtheorem{prop}[thm]{\bf{Proposition}}
\begin{document}

\renewcommand{\theequation}{\arabic{section}.\arabic{equation}}
\def\t{mathbb{T}}
\def\r{\mathbb{R}}
\def\H{\mathcal{H}}
\def\x{\mathcal{X}}
\def\y{\mathcal{Y}}
\def\C{\mathbb{C}}
\def\n{\mathbb{N}}
\def\Z{\mathbb{Z}}
\def\<{\langle}
\def\>{\rangle}
\def\D{\boldsymbol{\mathcal{D}}}
\def\F{\mathcal{F}}
\def\2{L^2}
\def\c{L^2([0,T];H^1(0,1))}
\def\i{\infty}
\def\h{\hat}
\def\v{\vec}
\def\lam{\lambda}
\def\ds{\displaystyle}

\renewcommand{\abstractname}{}
\renewcommand{\theequation}{\thesection.\arabic{equation}}
\numberwithin{equation}{section}
\theoremstyle{definition}
\renewcommand{\proofname}{\noindent Proof}
\newtheorem{Definition}{\noindent Definition}[section]
\renewcommand{\refname}{\center{References}}

\title{\Large On the High Dimensional RSA Algorithm$\raisebox{0.5mm}{\bf{{---}}}$A Public Key Cryptosystem Based on Lattice and Algebraic Number Theory}

\author{Zhiyong Zheng$^{~a}$ ~~ Fengxia Liu$^{~b,*}$ \\[10pt]
\small {$^{a,b }$  Engineering Research Center of Ministry of Education for Financial Computing} \\
\small {and Digital Engineering, Renmin University of China,}\\
\small { Beijing, 100872, P. R. China}\\
{*Corresponding author E-mail addresses:  liu$\_$fx@ruc.edu.cn}}

\date{}
\maketitle

    \noindent \textbf{Abstract}   The most known of public key cryptosystem was introduced in 1978 by Rivest, Shamir and Adleman\cite{R}
    and now called the RSA public key  cryptosystem  in their honor. Later, a few authors gave a simply extension of RSA over algebraic numbers field( see \cite{T}-\cite{U2}), but they require that the ring of algebraic integers is Euclidean ring, this requirement is much more stronger than the class number one condition.  In this paper, we introduce a high dimensional form of RSA by making use
   of the ring of algebraic integers of an algebraic number field and the lattice theory. We give an  attainable algorithm (see Algorithm I below) of which is significant both from the theoretical and practical point of view. Our main purpose in this paper is to show that the high dimensional RSA is a lattice based on public key         cryptosystem  indeed, of which would be considered as a new number in the family
    of post-quantum cryptography(see \cite{P} and \cite{P2}). On the other hand, we give a matrix expression for any algebraic number fields (see Theorem \ref{th2} below), which is a new result even in  the sense of classical algebraic number theory.

\vspace{0.4cm}

\newpage
\numberwithin{equation}{section}
\section{Introduction}

Let $ Q, \mathbb{R},\mathbb{C}$ be the rational numbers field, real numbers field,
and complex numbers field respectively, $ \mathbb{Z}$ be the integers ring. Let $ E \subset \mathbb{C}$ be an algebraic numbers
field of degree $n$, $R \subset E$ be the ring of algebraic integers of $E$. Suppose that $A \subset R$ is a non-zero ideal(all ideals in this paper are non-zero), then the factor ring $R/A$ is a finite ring, we denote by $N(A)$ the number of elements of $R/A$, which is called the norm of $A$,
and denote by $ \varphi(A)$ the number of invertible elements of $R/A$, which is called the Euler totient function of $A$. For any
$\alpha \in R,$ the principal ideal generated by $\alpha$ is denoted by $\alpha R$, then $\alpha$ is an invertible element of $R/A$ if
and only if $(\alpha R,A)=1.$ It is known (see Theorem 1.19 of \cite{N})that
\begin{equation}\label{1.1}
\varphi(A)=N(A)\prod_{P | A}(1-\frac{1}{ N(P)})
\end{equation}
where the product is extended over all prime ideals $P$ dividing $A$. Moreover, if $\alpha \in R$ and $ (\alpha R,A)=1,$ then
\begin{equation}\label{1.2}
\alpha^{\varphi(A)}\equiv 1(\text{mod}\  A).
\end{equation}

To generalize that RSA to arbitrary algebraic number fields $E$, we first show that the following assertion.

\begin{thm}\label{th1}
Let $P_{1}$ and $P_{2}$ be two distinct prime ideals of $R$ and $A=P_{1}P_{2}$, then for any $\alpha \in R$ and integer $k\geq 0,$ we have
\begin{equation}\label{1.3}
\alpha^{k\varphi(A)+1}\equiv \alpha(\operatorname{mod} \  A).
\end{equation}

\end{thm}
\begin{proof}
Let $ \alpha \in R.$ If $(\alpha R, A)=1,$ then  (\ref{1.3}) follows directly from (\ref{1.2}). If $(\alpha R, A)=A,$ then $ \alpha R \subset  A$ and $ \alpha \in A,$ (\ref{1.3}) is trivial. Thus, we only consider the cases of $(\alpha R, A)=P_{1}$ and $(\alpha R, A)=P_{2}$. If
$ (\alpha R, A)=P_{1}$, then $ (\alpha R, P_{2})=1$, by (\ref{1.2}) we have
$$
\alpha^{\varphi(P_{2})}\equiv 1(\text{mod}\  P_{2}).
$$
It follows that
$$
\alpha^{k \varphi(A)}\equiv 1(\text{mod}\  P_{2}),\ \ \forall k \in \mathbb{Z},\ \ k\geq 0.
$$
Therefore, there exists an element $\beta \in P_{2}$ such that
$$
\alpha^{k \varphi(A)}= 1+ \beta.
$$
We thus have
$$
\alpha^{k \varphi(A)+1}= \alpha+ \alpha\beta, \ \ \text{and} \ \ \alpha^{k \varphi(A)+1}\equiv \alpha(\text{mod}\ A),
$$
since $ \alpha\beta \in A.$ The same reason gives (\ref{1.3}) when $ (\alpha R, A)=P_{2}$.

\end{proof}

According to Theorem \ref{th1}, one can easily extend the classical RSA over an algebraic number field as follows(also see \cite{T} , but it does not give the proof of (\ref{1.3}) ).

\newpage
\begin{table}[h]
\begin{tabular}{l}
\hline
 { \bf{ \  \  RSA in the ring of algebraic integers}}\\
\hline
\\
$\bullet $\ \ {\bf{\text{Parameters:}}} \ \  \  $n\geq 1$ \text{is a positive integer}, $E/Q$ is an algebraic numbers field of\\
\ \ \ \ \ \ \ \ \ \ \ \ \ \ \ \ \  \ \ \ \ \  \ degree $n$,
$ R\subset E$  \text{ is the ring of algebraic integers of } $E$. $P_{1}$ \text{and}  $P_{2}$ \\
\ \ \ \ \ \ \ \ \ \ \ \ \ \ \ \ \  \ \ \ \ \  \ are two prime ideals  of  $R$,
 $A=P_{1}P_{2}$,  $R/A $ is the factor ring,\\
  \ \ \ \ \ \ \ \ \ \ \ \ \ \ \ \ \  \ \ \ \ \ \
$ S $ is a set of coset representatives   of R/A, $\varphi(A)$  is the Euler \\
  \ \ \ \ \ \ \ \ \ \ \ \ \ \ \ \ \  \ \ \ \ \ \  function of  $A$,
 $1\leq e < \varphi(A)$
\text{ and} $1\leq d < \varphi(A)$ are two positive\\
      \ \ \ \ \ \ \ \ \ \ \ \ \ \ \ \ \  \ \ \ \ \  \    integers such that
   $ed\equiv 1(\text{mod}\ \ \varphi(A))$. \\
$\bullet$\ \ {\bf{Public keys:}} \ \   \text{ The ideal}  $A $ \text{and positive integer} $e $\text{ are the public keys}.\\
$\bullet$\ \ {\bf{Private keys:}} \ \   \text{The prime ideals } $P_{1}, P_{2}$ \text{ and the positive integer} $d$  \text{are the }\\
  \ \ \ \ \ \ \ \ \ \ \ \ \ \ \ \ \  \ \ \ \ \  \ \  private keys.\\
$\bullet$\ \  {\bf{Encryptions:}}  \ \  \text{For any input message}  $ \alpha \in S $, \text{ the ciphertext} $c$  \text{is} $c\equiv\alpha^{e}(\text{mod}\ \  A)$.\\
$\bullet$\ \  {\bf{Decryption:}} \ \  \ $c^{d} \equiv \alpha^{e d} \equiv \alpha(\text{ mod} \ \ A )$, \text{one can find plaintext} $\alpha$ \text{from} $c$ {in} $S$.\\
\hline
\end{tabular}
\caption{RSA in the ring of algebraic integers}
\end{table}

Obviously, if $n=1$, the above algorithm is  the ordinary RSA. However, it is difficult to find the prime ideals in $R$ and to construct a set of coset representatives of $R/A$ yet. In \cite{T}, the author supposed the ring $R$ is Euclidean ring, so that $S$ can be constructed by Euclidean algorithm in $R$.  The simplest way is to select an prime element $\alpha$ in $R$, so that the principal ideal $\alpha R$ is a prime ideal. In algorithm  I, we would precisely construct a set of coset representatives for the factor ring $R/A$ by  the lattice theory. Here we give an approximately construction of the set of coset representatives
for factor ring $R/A$.

If $P \subset R$ is a prime ideal, then $P \cap \mathbb{Z}=p \mathbb{Z}$, where $p \in \mathbb{Z}$ is a rational prime number. Since $R/P$ is a finite field and $\mathbb{Z}/(p\mathbb{Z})\subset  R/P$, thus $N(P)=p^{f}$, where $f\left(1 \leq f \leq n\right)$ is called the degree of $P$. We write $p R=P_{1}^{e_{1}} P_{2}^{e_{2}} \cdots P_{g}^{e_{g}}$, where $P=P_{1}$ and $P_{i}$ are distinct prime ideals, $e_{i}$ is called the ramification index of $P_{i}$. There exists a remarkable relation among ramification indexes and degrees (see Theorem 3 of page 181 of \cite{I})
\begin{equation}\label{1.4}
\sum_{i=1}^{g} e_{i} f_{i}=n.
\end{equation}
Let $\left\{\alpha_{1}, \alpha_{2}, \cdots \alpha_{n}\right\} \subset R$ be an integral basis for $E/Q, A=P_{1} P_{2}$. Suppose that $P_{1} \cap \mathbb{Z}=p\mathbb{Z}$ and $P_{2} \cap \mathbb{Z}= q\mathbb{Z}$,  then $A \cap \mathbb{Z}= pq\mathbb{Z}$, where $p$ and $q$ are two distinct rational prime numbers.
\begin{lem}\label{lm1.1}

Let
\begin{equation}\label{1.5}
S_{1}=\left\{\sum_{i=1}^{n} a_{i} \alpha_{i} \mid 0 \leq a_{i}<p q, \ a_{i} \in \mathbb{Z},\ 1 \leq i \leq n \right\}.
\end{equation}
Then $S_{1}$ covers a set of coset representatives of $R/A$. Moreover, if the degrees of $P_{1}$  and $P_{2}$ are $n$, then $S_{1}$ is precisely an set of coset representatives of $R/A$.

\end{lem}

\begin{proof}

 Since $A=P_{1}P_{2}$,  $P_{1} \cap \mathbb{Z}=p\mathbb{Z}$ and $P_{2} \cap \mathbb{Z}=q\mathbb{Z}$, we have $pq R \subset A$, thus $R/pqR $ maps onto $R/A$. To prove the first assertion, it is enough to show that $S_{1}$ is a set of coset representatives of $R/pq R$. Since $\left\{\alpha_{1}, \alpha_{2}, \ldots \alpha_{n}\right\}$ is an integral  basis and
$$
R=\mathbb{Z}\alpha_{1}+\mathbb{Z} \alpha_{2}+\cdots +\mathbb{Z} \alpha_{n}.
$$
Suppose that $\alpha=\sum_{i=1}^{n} m_{i} \alpha_{i} \in R$, write $m_{i}=a_{i} pq+r_{i}$, where $0 \leq r_{i}<pq$. Clearly
$$
\alpha \equiv \sum_{i=1}^{n} r_{i} \alpha_{i}(\text{mod}\ pqR).
$$
Thus every coset of $ pqR$ contains an element of $ S_{1}$. If   $\sum_{i=1}^{n} r_{i}\alpha_{i}=\sum_{i=1}^{n} r'_{i}\alpha_{i} $ are in $S_{1}$ and in the same coset mod $pqR$, then
$$
\sum_{i=1}^{n}\left(r_{i}-r_{i}^{\prime}\right) \alpha_{i} \equiv 0 (\text{mod}\ \ pqR).
$$
Since $ \alpha_{i}$ are linearly independent, it follows that
$$
r_{i}\equiv r_{i}^{\prime} (\text { mod}\ \  pq)\ \ \text { and } \ \ r_{i}=r_{i}^{\prime},\ \  1 \leq i \leq n .
$$
Next, suppose that the degrees of $P_{1}$ and $P_{2}$ are $n$, then $N\left(P_{1}\right)=p^{n}$ and $N\left(P_{2}\right)=q^{n}$, by (\ref{1.4}) we thus have  $P_{1}=p R$,  $P_{2}=q R$ and $A=pq R$. The second assertion  follows immediately.
\end{proof}

If one replaces $S$ by $S_{1}$ in Table 1, then the successful probability of  decryption is
\begin{equation}\label{1.6}
N(A) / p^{n} q^{n}=p^{f_{1}-n} q^{f_{2}-n},
\end{equation}
where  $f_{1}$ and $f_{2}$ are the  degrees  of $P_{1}$ and $P_{2}$ respectively.

We note that $f_{1}=f_{2}=n$ if and only if $P_{1}=p R$ and
$P_{2}=q R$, in this special case, we may give a numerical explanation. It is easy to see that
$$
\varphi(A)=\varphi(p R) \varphi(q R)=\left(p^{n}-1\right)\left(q^{n}-1\right).
$$
By Theorem \ref{th1}, for any $ a \in \mathbb{Z}$, we have
\begin{equation}\label{1.7}
a^{k\left(p^{n}-1\right)\left(q^{n}-1\right)+1} \equiv a (\text{mod}\ pq),\ \ k \in \mathbb{Z},\ \ k\geq 0.
\end{equation}
Since $S_{1}$ is a set of coset representatives of $R/A,$ $\alpha=\sum_{i=1}^{n}a_{i}\alpha_{i}\in S_{1}$,
We may regard  $\alpha$ as a vector  $\left(a_{1}, a_{2}, \ldots, a_{n}\right) \in \mathbb{Z}_{pq}^{n}$. Let $m =pq$, $1 \leq e<\left(p^{n}-1\right)\left(q^{n}-1\right)$ and $1 \leqslant d<\left(p^{n}-1\right)\left(q^{n}-1\right)$ such that
$$
ed\equiv 1(\text{mod}\ \ (p^{n}-1)(q^{n}-1)).
$$
Then for every input message $\alpha=\left(a_{1}, a_{2}, \cdots,a_{n}\right)$, we use the public key $(m,e) $ and private key $(p,q, d)$  to
 encryption and decryption for each  $a_{i}$ in order, obviously, this is the algorithms given by \cite{T}, we consider these algorithms
 are just a simply repeat of RSA.

The main purpose of this paper is to show that the high dimensional form of  RSA algorithm is a lattice based on cryptosystem in general. To do this, we first establish a relationship  between an algebraic number field $E$ and the Euclidean space $Q^{n}$. Let $\mathbb{R}^{n}$ be the Euclidean space of which is a linear space over $\mathbb{R}$  with the Euclidean  norm $|x|$,
\begin{equation}\label{1.8}
|x|=\left(\sum_{i=1}^{n} x_{i}^{2}\right)^{\frac{1}{2}},\ \  \text { where }\  x^{\prime}=\left(x_{1}, x_{2}, \cdots, x_{n}\right) \in \mathbb{R}^{n}.
\end{equation}
We use the column natation for vector in $\mathbb{R}^{n}$, and $x^{\prime}$ is the transpose  of $x$, which is called a row vector in $\mathbb{R}^{n}$. $Q^{n} \subset \mathbb{R}^{n}$ is a subspace of $\mathbb{R}^{n}.$

Without loss of  generality, an algebraic number  field $E$ of degree $n$ may express as $E=Q(\theta)$, where $\theta$ is an algebraic integer of degree $n$  and $Q(\theta)$ is the field generated by $\theta$ over $Q$.  Let $\phi(x)$ be the minimal polynomial of $\theta$,
\begin{equation}\label{1.9}
\phi(x)=x^{n}-\phi_{n-1} x^{n-1}-\cdots-\phi_{1} x-\phi_{0} \in \mathbb{Z}[x],
\end{equation}
where all $\phi_{i} \in \mathbb{Z}.$ It is known that
\begin{equation}\label{1.10}
E=Q[\theta]=\left\{\sum_{i=0}^{n-1} a_{i} \theta^{i} \mid a_{i} \in Q \right\} .
\end{equation}

We define an one to one correspondence between $E$ and $Q^{n}$ by $\tau$:
\begin{equation}\label{1.11}
\alpha=\sum_{i=0}^{n-1} a_{i} \theta^{i}  \in E \stackrel{\tau}{\longrightarrow}
\overline{\alpha}=\begin{pmatrix}
     a_0 \\
    a_1  \\
     \vdots \\
     a_{n-1}
\end{pmatrix}
\in Q^{n}
\end{equation}
and write $\tau(\alpha)=\overline{\alpha}$ or $\alpha \stackrel{\tau}{\rightarrow} \overline{\alpha}$. In fact $\tau$ is a  homomorphism of additive group from $E$ to $Q^{n}$, because of $\tau(a \alpha)=a \tau(\alpha)$ for all $a\in \mathbb{Q}.$

As usual, the trace and norm mappings from $E$ to $Q$ are denoted by
$$
\operatorname{tr}(\alpha)=\operatorname{tr}_{E/Q}(\alpha),\ \  \text { and } \ \ N(\alpha)=N_{E/Q}(\alpha).
$$
It  is known (see corollary of page 58 of  \cite{N}) that
\begin{equation}\label{1.12}
N(\alpha R)=|N(\alpha)|, \quad \forall \alpha \in R .
\end{equation}
A full rank  lattice $L$ is a discrete addition subgroup of $\mathbb{R}^{n}$, the equivalent expression for $L$ is   ( See \cite{M2} and \cite{Z})
\begin{equation}\label{1.13}
L=L(B)=\left\{B x \mid x \in \mathbb{Z}^{n}\right\} ,
\end{equation}
where $B=\left[\overline{\beta}_{1}, \overline{\beta}_{2}, \cdots, \overline{\beta}_{n}\right]_{n \times n} \in \mathbb{R}^{n \times n}$ is an invertible matrix of $n\times n$ dimension, $B$ is called a generated matrix of $L$. If $L \subset Q^{n}$, we call $L$  a rational lattice, if $L\subset \mathbb{Z}^{n}$, we call $L$  an integer lattice. It is not difficult to see that every ideal of $R$ corresponds an rational lattice, we have

\begin{lem}\label{lm1.2}
Let $A \subset R$ be an ideal and $A \neq 0$, then $\tau(A)$ is a rational lattice.

\end{lem}
\begin{proof} Let $\left\{\beta_{1}, \beta_{2},\cdots, \beta_{n}\right\} \subset A$ be an integral basis for $E/Q$,
one has
$$
A=\mathbb{Z} \beta_{1}+\mathbb{Z} \beta_{2}+\cdots+\mathbb{Z} \beta_{n}.
$$
It follows that
$$
\tau(A)=\mathbb{Z} \overline{\beta}_{1}+\mathbb{Z}\overline{ \beta}_{2}+\cdots+\mathbb{Z}\overline{ \beta}_{n},
$$
where $ \overline{\beta}_{i}= \tau ( \beta_{i}) \in Q^{n}$. Let $B=[\overline{\beta}_{1},\overline{\beta}_{2},\cdots,\overline{\beta}_{n} ]$, since $ \{ \beta_{1},\beta_{2},\cdots,\beta_{n}\}$  is linearly independent over $Q$, thus $B$ is an invertible matrix, and we have
$$
\tau(A)=L(B)=\{Bx \mid x\in \mathbb{Z}^{n}\}.
$$
The lemma follows at once.
\end{proof}

 Let $L \subset Q^{n}$ is a rational lattice, of which be corresponded by an ideal $A$ in $E$ for some suitable algebraic number field $E$, we call $L$  an ideal lattice. Ideal lattice was first introduced  by Lyubashevsky and Miccancio in \cite{L} in the case of integer lattice,
 here we generalize this notation to the case of rational lattices.  More detail discussion about ideal lattice, we refer to \cite{Z}.

 To give an attainable algorithm for high dimensional RSA, we require the following NC-property for the algebraic number  field E.
\begin{equation}\label{1.14}
\text{NC- property: } \  \ \ \ \ \  \ \  \ \ \ \ \ \  E=Q(\theta) \ \ \ \text { and } \ \ \ \ R=\mathbb{Z}[\theta],\ \ \ \ \ \ \ \ \ \ \ \ \   \ \ \ \ \ \ \ \ \ \ \ \ \ \ \ \
\ \ \ \ \ \ \ \ \ \ \ \ \ \ \ \ \ \   \ \ \ \ \
\end{equation}
where
\begin{equation}\label{1.15}
\mathbb{Z}[\theta]=\{ \sum_{i=0}^{n-1}a_{i}\theta^{i} \mid a_{i}\in \mathbb{Z}, \ 1\leq i\leq n\}.
\end{equation}

Some of well-known algebraic number fields  satisfy the NC-property, we list a few as follows.
\begin{table}[h]
\begin{tabular}{l}
\hline
 $${ \text{\bf{ \text{Algebraic number fields with NC-property}}}} $$\\
\hline
\\
$\bullet $\ \ \text{Quadratic Fields(see Proposition 13.1.1 of \cite{I} )}:  \\
\ \ \ \ \ \ \ \ \ \ \ \ \ \ \ \ \  \ \ \ \ \  \ \  $E=Q(\sqrt{d}),$ where $ d\in\mathbb{Z} $
is a square-free integer
and $d=2,3(\text{mod}\ \ 4).$
\\
$\bullet$\ \  Cyclotemic Fields (see theorem 2.6 of \cite{W} ): \\
\ \ \ \ \ \ \ \ \ \ \ \ \ \ \ \ \  \ \ \ \ \  \ \  $E=Q\left(\xi_{n}\right)$,
where $\xi_{n}=e^{2 \pi i /n}$ is a primitive  $n$-th root  of unity.\\
$\bullet $\ \  Totally Real Algebraic Number Fields (see Proposition 2.16 of \cite{W} ): \\ \ \ \ \ \ \ \ \ \ \ \ \ \ \ \ \ \ \ \  \ \ \ \ \
$E=Q(\xi_{n}+\xi_{n}^{-1}), $    and $E \subset \mathbb{R}$ is the maximal real subfield of $Q(\xi_{n} )$.\\
\hline
\end{tabular}
\caption{Algebraic number fields with NC-property}
\label{Table 1}
\end{table}

\section{Ideal Matrices}

Suppose that $\theta$ is an algebraic integer of degree $n$, $ \phi(x)=x^{n}-\phi_{n-1} x^{n-1}-\cdots-\phi_{1} x-\phi_{0} \in \mathbb{Z}[x]$ is the minimal polynomial of $\theta$, thus $\phi(x)$ is irreducible. Let $ \theta=\theta_{0}, \theta_{1}, \theta_{2},\cdots,\theta_{n-1}$ be $n$ different roots of $\phi(x)$, the Vandermonde matrix of $\phi(x)$ is defined by
 \begin{equation}\label{2.1}
 V=V_{\phi}=\left[\theta_{j}^{i}\right]_{0 \leq i, j \leq n-1}, \ \ \text{and}\ \ \Delta=\text{det}( V_{\phi}) \neq 0 .
\end{equation}

According to $\phi(x)$, we denote  the rotation matrix or adjoint matrix (see page 116 of \cite{M}) by
\begin{equation}\label{2.2}
 H=H_{\phi}=
\left(
\begin{array}{ccc|c}
0 & \cdots & 0 & \phi_0\\
\hline
& & & \phi_1\\
& I_{n-1} & & \vdots \\
& & & \phi_{n-1} \\
\end{array}
\right)
\in \mathbb{Z}^{n\times n},
\end{equation}
where $ I_{n-1} $  is the unit matrix of $n-1$ dimension.

\begin{Definition}\label{df2.1}
An  ideal  matrix $H^{*}(\overline{f})$  generated by the input vector $\overline{f} \in \mathbb{R}^{n} $ is defined by
  \begin{equation}\label{2.3}
 H^{*}(\overline{f})=\left[ \ \overline{f}, H\overline{f}, \cdots, H^{n-1}\overline{f} \ \right]_{n \times n} \in \mathbb{R}^{n\times n}
\end{equation}
  and all ideal matrices are  denoted by
\begin{equation}\label{2.4}
M_{\mathbb{R}}^{*}=\left\{H^{*}(\overline{f}) \mid \overline{f} \in \mathbb{R}^{n}\right\}\ \  \text { and } \ M_{Q}^{*}=\left\{H^{*}(\overline{f}) \mid \overline{f} \in Q^{n}\right\}.
\end{equation}
\end{Definition}

\begin{Definition}\label{df2.2}  For any two vectors $\overline{f}$ and $\overline{g}$ in $\mathbb{R}^{n}$, the $\phi$-conventional product is defined by
\begin{equation}\label{2.5}
\overline{f} \otimes \overline{g}=H^{*}(\overline{f} ) \overline{g}
\end{equation}
and   the m-multi product is denoted by
\begin{equation}\label{2.6}
\overline{f}^{\otimes m} =\overbrace{\overline{f} \otimes \overline{f} \otimes \cdots \otimes \overline{f}}^{m},\ \  m \in \mathbb{Z}, \ \ m \geq 1 .
\end{equation}
\end{Definition}

 \begin{rem}\label{rm1}
 If $ \phi(x)=x^{n}-1$, then $H_{\phi}$ is the classical circulant matrix (see \cite{D-a}), conventional product with circulant matrix was first proposed by Hoffstein, Pipher and Silverman in \cite{H}, which  plays  a key role in their cryptosystem. In \cite{Z2},  we generalized this definition with more general rotation matrices.
  \end{rem}

  By (\ref{2.3}), $H^{*}(\overline{f})=0$ is a zero  matrix if and only if $\overline{f}=0$ is a zero vector, and $H^{*}(\overline{f}+\overline{g})=H^{*}(\overline{f})+H^{*}(\overline{g})$, then $H^{*}(\overline{f})=H^{*}(\overline{g})$ if and only if $\overline{f}=\overline{g}$. Thus we may regard $H^{*} :\mathbb{R}^{n} \rightarrow \mathrm{M}_{\mathbb{R}}^{*}$ as an one to one correspondence, which is also a homomorphism of Abel group.

The main aim of this subsection is to show the $Q^{n}$ is a field under the $\phi$-conventional product and $M_{Q}^{*}$ is also a
field under the ordinary additive and product of matrices, both of them are isomorphic to the algebraic number field $E=Q(\theta)$.
To do this, we require some basic properties of the ideal  matrices.

Let $\overline{e}_{1}, \overline{e}_{2}, \cdots, \overline{e}_{n}$ be the unit vectors of  $\mathbb{R}^{n}$, namely
\begin{equation}\label{2.7}
\overline{e}_{1}=\left(\begin{array}{c}
1 \\
0\\
\vdots \\
0
\end{array}\right), \overline{e}_{2}=\left(\begin{array}{l}
0 \\
1 \\
\vdots \\
0
\end{array}\right), \cdots, \quad \overline{e}_{n}=\left(\begin{array}{l}
0 \\
0 \\
\vdots \\
1
\end{array}\right) .
\end{equation}
\begin{lem}\label{lm2.1}
 Let $\tau$ be defined by (\ref{1.11}), then we have
\begin{equation}\label{2.7a}
\left\{\begin{array}{lll}
\tau\left(\theta^{k}\right)=\overline{e}_{k+1}, \ \ 0 \leq k \leq n-1 \\
H^{*}\left(\overline{e}_{k}\right)=H^{k-1}, \ \ 1 \leqslant k \leqslant n .
\end{array}\right.
\end{equation}
\end{lem}
\begin{proof}

$\quad \tau\left(\theta^{k}\right)=\overline{e}_{k+1}$ follows directly from the definition of $\tau$. We use induction to prove $H^{*}\left(\overline{e}_{k}\right)=H^{k-1}$.  It is easy to see that
$H^{*}\left(\overline{e}_{1}\right)=I_{n}$, the unit matrix of $n$ dimension.  Suppose that $H^{*}\left(\overline{e}_{k-1}\right)$ $=H^{k-2}$, for $k\geq 2$, note that $\overline{e}_{k}=H \overline{e}_{k-1}$, it follows that
\begin{align*}
H^{*}\left(\overline{e}_{k}\right)&=\left[H \overline{e}_{k-1}, H^{2} \overline{e}_{k-1}, \cdots, H^{n} \overline{e}_{k-1}\right]\\
&= H\left[ \overline{e}_{k-1}, H \overline{e}_{k-1}, \cdots, H^{n-1} \overline{e}_{k-1}\right]\\
&=HH^{*}(\overline{e}_{k-1})= HH^{k-2}=H^{k-1}.
\end{align*}
The lemma follows immediately.
\end{proof}

Since $\phi(x)$ is the characteristic polynomial of $H$, by Hamilton-Cayley theorem, we have
\begin{equation}\label{2.9}
\phi(H)=0, \ \ \text{or}\  H^{n}=\phi_{0}+\phi_{1} H+\cdots+\phi_{n-1} H^{n-1}.
\end{equation}
Therefore, all the rotation matrices $H^{k}(k \geq 0)$ are the ideal matrices, especially, the unit matrix $I_{n}=H^{*}\left(\overline{e}_{1}\right)$ is an ideal matrix.

Let $\mathbb{R}[x]$ be the polynomials ring and $\mathbb{R}(x) /\langle\phi(x)\rangle$ be the quotient ring, where $\langle\phi(x)\rangle$ is the principal ideal generated by $\phi(x)$ in $\mathbb{R}[x]$. We establish an one to one correspondence $t$ between $\mathbb{R}^{n}$ and $\mathbb{R}[x]/ \langle\phi(x)\rangle$ by
\begin{equation}\label{2.10}
\overline{f}=\left(\begin{array}{c}
f_{0} \\
f_{1} \\
\vdots \\
f_{n-1}
\end{array}\right)\in \mathbb{R}^{n} \stackrel{t}{\longrightarrow} f(x)=f_{0}+f_{1} x+\cdots+f_{n-1} x^{n-1} \in \mathbb{R}[x]/ \langle\phi(x)\rangle
\end{equation}
and write $t(\overline{f})=f(x)$, or $t^{-1}(f(x))=\overline{f}$.

\begin{lem} \label{lm2.2}
For any $\overline{f} \in \mathbb{R}^{n}$, the ideal matrix $H^{*}(\overline{f})$ is given by
 \begin{equation}\label{2.11}
 H^{*}( \overline{f})=f(H)=f_{0}I_{n}+f_{1}H+\cdots+f_{n-1}H^{n-1}.
 \end{equation}
 Moreover, if $F(x) \in \mathbb{R}[x]$ and $F(x)\equiv f(x)(\text{mod}\ \phi(x))$, then $f(H)=F(H)$.
 \end{lem}
 \begin{proof}
 Writing $\overline{f}=f_{0} \overline{e}_{1}+f_{1} \overline{e}_{2}+\cdots+f_{n-1} \overline{e}_{n}$, by Lemma \ref{lm2.1}, we have
  \begin{align*}
 H^{*}( \overline{f})&=f_{0}H^{*}( \overline{e}_{1})+f_{1}H^{*}( \overline{e}_{2})+\cdots+f_{n-1}H^{*}(\overline{ e}_{n})\\
 &=f_{0}I_{n}+f_{1}H+\cdots+f_{n-1}H^{n-1}=f(H) .
 \end{align*}
 Suppose that $ F(x) \equiv f(x)(\text{mod}\ \phi(x))$, by (\ref{2.9}) we have $f(H)=F(H)$ immediately.

 \end{proof}

\begin{lem} \label{lm2.3}  Let $\overline{f}$ and $\overline{g}$ be two vectors in $\mathbb{R}^{n}$, and $f(x), g(x)$
be the corresponding polynomials respectively, then we have
 \begin{equation}\label{2.12}
  t(\overline{f} \otimes \overline{g}) \equiv f(x) g(x) (\text{mod}\ \  \phi(x)) .
\end{equation}
 \end{lem}
\begin{proof} Since $t$ is a bijection, it is suffice to show that
\begin{equation}\label{2.13}
t^{-1}(f(x) g(x))=\overline{f} \otimes \overline{g}.
\end{equation}
 Let $g(x)=g_{0}+g_{1} (x)+\cdots+g_{n-1} x^{n-1} \in \mathbb{R}[x] /\langle \phi(x)\rangle$, then
\begin{align*}
x g(x)&=g_{0} x+\cdots+g_{n-1}x^{n}\\
&=g_{n-1}\phi_{0}+(g_{0}+\phi_{1}g_{n-1} )x+\cdots +(g_{n-2}+\phi_{n-1}g_{n-1}  )x^{n-1}.
\end{align*}
It follows that
$$
t^{-1}(x g(x))=H t^{-1}(g(x))=H \overline{g}.
$$
More general, we have
\begin{equation}\label{2.13}
t^{-1}\left(x^{k} g(x)\right)=H^{k} t^{-1}(g(x))=H^{k}\overline{g},\quad 0 \leq  k \leq n-1.
\end{equation}
Let $f(x)=f_{0}+f_{1} x+\cdots+f_{n-1} x^{n-1}$, then
$$
\begin{aligned}
t^{-1}(f(x) g(x)) =\sum_{k=0}^{n-1} f_{k} t^{-1}\left(x^{k} g(x)\right) =\sum_{k=0}^{n-1} f_{k} H^{k} \overline{g} =H^{*}(\overline{f}) \overline{g}= \overline{f}\otimes \overline{g}.
\end{aligned}
$$
The lemma follows immediately.
\end{proof}

\begin{lem} \label{lm2.4}
For any two vectors $
\overline{f}=\left(\begin{array}{c}
f_{0} \\
f_{1} \\
\vdots \\
f_{n-1}
\end{array}\right)\in \mathbb{R}^{n},\ \ \overline{g}=\left(\begin{array}{c}
g_{0} \\
g_{1} \\
\vdots \\
g_{n-1}
\end{array}\right)\in \mathbb{R}^{n},
$
we have the following properties for ideal matrices:

 (\text{i}) \ \ $H^{*}(\overline{f}) H^{*}(\overline{g})=H^{*}\left(\overline{g}) H^{*}(\overline{f}\right) ;$

 (\text{ii}) \ \ $H^{*}(\overline{f}) H^{*}(\overline{g})=H^{*}( H^{*}(\overline{f}) \overline{g}) ;$

 (\text{iiii}) \ \  $H^{*}(\overline{f})=V_{\phi}^{-1} \operatorname{diag}\left\{f\left(\theta_{0}\right), f\left(\theta_{1}\right), \cdots, f\left(\theta_{n-1}\right)\right\} V_{\phi}$;

 (\text{iv})\ \  $\operatorname{det}\left(H^{*}(\overline{f})\right)=\prod_{i=0}^{n-1} f\left(\theta_{i}\right)$;

 (\text{v})\ \  If $\overline{f} \in Q^{n}$,  $\overline{f} \neq 0$, then $H^{*}(\overline{f})$ is an invertible matrix and
$$
\left(H^{*}(\overline{f})\right)^{-1}=H^{*}(\overline{u}),
$$
where $u(x) \in Q[x]$ is the unique polynomial such that
$u(x) f(x) \equiv 1 (\text{mod}\  \phi(x))$ in $Q[x]$.
\end{lem}

\begin{proof}
By Lemma \ref{lm2.2}, we have
$$
H^{*}(\overline{f})  H^{*}(\overline{g})=f(H) g(H)=g(H) f(H)=H^{*}(\overline{g})  H^{*}(\overline{f}) .
$$
To  prove (ii), we write $H^{*}(\overline{f}) \overline{g}=\overline{f} \otimes \overline{g}$, it follows that
$$
H^{*}\left(H^{*}(\overline{f}) \overline{g}\right)=H^{*}(\overline{f} \otimes \overline{g})=f(H) g(H)=H^{*}(\overline{f}) \cdot H^{*} ( \overline{g}).
$$
By theorem 3.5 of \cite{D-a}, we have
\begin{equation}\label{2.15}
H=V_{\phi}^{-1} \operatorname{diag}\left\{\theta_{0}, \theta_{1}, \cdots, \theta_{n- 1}\right\} V_{\phi}
\end{equation}
It follows that
$$
H^{*}(\overline{f})=f(H)=V^{-1} _{\phi}\operatorname{diag}\left\{f\left(\theta_{0}\right), f\left(\theta_{1}\right), \cdots, f\left(\theta_{n-1}\right) \right\}  V_{\phi}.
$$
Since $\operatorname{diag}\left\{f\left(\theta_{0}\right), f\left(\theta_{1}\right), \cdots ,f\left(\theta_{n-1}\right)\right\}$ is a diagonal matrix, we have
$$
\begin{aligned}
\operatorname{det}\left(H^{*}(\overline{f})\right) =\operatorname{det}\left(\operatorname{diag}\left\{f\left(\theta_{0}\right), f\left(\theta_{1}\right), \cdots, f\left(\theta_{n-1}\right)\right\}\right) =\prod_{i=0}^{n-1} f\left(\theta_{i}\right) .
\end{aligned}
$$
To show that the last assertion, since $\overline{f} \in Q^{n}, \overline{f} \neq 0$, and $\phi(x)$ is an irreducible polynomial, thus we have $(f(x), \phi(x))=1$ in $Q[x]$, There are $u(x) \in Q[x]$ and $v(x) \in Q[x]$ such that
$$
u(x) f(x)+v(x) \phi(x)=1.
$$
By (\ref{2.13}) and noting that $t^{-1}(1)=\overline{e}_{1} \in \mathbb{R}^{n}$, we have $\overline{u} \otimes \overline{f}=\overline{e}_{1}$. It follows that
$$
H^{*}(\overline{u}) \cdot H^{*}(\overline{f})=H^{*}(\overline{e}_{1})=I_{n} .
$$
We complete the proof of Lemma.
\end{proof}

Next, we discuss the algebraic number field $E=Q(\theta)$, recall $\tau$ is an one to one correspondence between $E$ and $Q^{n}$.

\begin{lem} \label{lm2.5}
Far any two elements $\alpha$ and $ \beta$  in $E$, we have
\begin{equation}\label{2.16}
\tau( \alpha\beta)=\tau( \alpha)\otimes \tau( \beta)=\overline{\alpha}\otimes \overline{\beta}.
\end{equation}
\end{lem}
\begin{proof}
 Let $\beta=\beta_{0}+\beta_{1} \theta+\cdots+\beta_{n-1} \theta^{n-1}$, where $\beta_{i}\in Q$, it is easily seen that
$$
\theta \beta=\phi_{0} \beta_{n-1}+\left(\beta_{0}+\phi_{1} \beta_{n-1}\right) \theta+\cdots+\left(\beta_{n-2}+\phi_{n-1} \beta_{n-1}\right) \theta^{n-1},
$$
thus we have $\tau(\theta \beta)=H \tau(\beta)=H \overline{\beta}$, and
\begin{equation}\label{2.17}
\tau\left(\theta^{k}\beta\right)=H^{k} \tau ( \beta)=H^{k} \overline{\beta}, \quad 0 \leq k \leq n-1.
\end{equation}
Let $\alpha=\alpha_{0} + \alpha_{1} \theta+\cdots+\alpha_{n-1} \theta^{n-1}$, by lemma \ref{lm2.2}, we have
\begin{align*}
\tau(\alpha \beta)=\sum_{k=0}^{n-1} \alpha_{k}  \tau\left(\theta^{k} \beta\right)=\sum_{k=0}^{n-1} \alpha_{k} H^{k} \overline{\beta}=H^{*} ( \overline{\alpha})  \overline{\beta}=\overline{\alpha}\otimes \overline{\beta},
\end{align*}
the lemma follows immediately.
\end{proof}

Let $A=\left(a_{i j}\right)_{n \times n}$ be a square matrix, the trace of $A$ is defined by  $\operatorname{Tr}(A)=\sum_{i=1}^{n} a_{ ii}$ as usual. The main result of this subsection is the following theorem.

\begin{thm}\label{th2}
Let $E=Q(\theta)$ be an algebraic number field of degree $ n$, $\phi(x) \in \mathbb{Z}[x]$ be the minimal polynomial of $\theta$. Then the linear space $Q^{n}$ is a field under the $\phi$-conventional product, and all of ideal matrices $M_{Q}^{*}$ generated by rational vectors is also a field with the ordinary additive and product of matrices. Both of them are  isomorphic to $E$, namely
\begin{equation}\label{2.18}
E \cong Q^{n} \cong M_{Q}^{*} .
\end{equation}
Moreover, let $\alpha \in  E$,  $\operatorname{tr}(\alpha)$ and $N(\alpha)$ be the trace and norm of $\alpha$, then we have
\begin{equation}\label{2.19}
\operatorname{tr}(\alpha)=\operatorname{Tr}\left(H^{*}(\overline{\alpha})\right),\  \text { and } \ \ N(\alpha)=\operatorname{det}\left(H^{*}(\overline{\alpha})\right) .
\end{equation}
\end{thm}
\begin{proof}
$\tau: E \rightarrow Q^{n}$ given by $(\ref{1.11})$,  it is clearly that
$$
\tau(\alpha+\beta)=\tau(\alpha)+\tau(\beta),\  \ \text { and } \ \ \tau(\alpha \beta)=\tau(\alpha)\otimes\tau(\beta).
$$
Thus $Q^{n}$ is a field under the $\phi$-conventional product and $E \cong Q^{n}$.  By lemma \ref{lm2.4}, we have
$$
H^{*}(\overline{\alpha}+\overline{\beta})=H^{*}(\overline{\alpha})+H^{*}(\overline{\beta}) \ \ \text { and }
\ \ H^{*}\left(\overline{\alpha} \otimes \overline{\beta}\right)=H^{*}(\overline{\alpha}) H^{*}(\overline{\beta}),
$$
thus $M_{Q}^{*}$ is also a field and $E \cong Q^{n} \cong M_{Q}^{*}$.

 The main difficult is to prove (\ref{2.19}).
We observe that $\theta$ induces a linear  transformation of $E/Q$ by $\alpha \rightarrow \theta \alpha$, and the matrix of this linear
 transformation under basis $\left\{1, \theta, \theta^{2},\cdots,  \theta^{n-1}\right\}$ is just $H$, namely
$$
\theta\left(1, \theta, \theta^{2}, \cdots, \theta^{n-1}\right)=\left(1 ,\theta , \theta^{2}, \cdots, \theta^{n -1}\right) H .
$$
By the definition of trace, we have
$$
\operatorname{tr}(\theta)=\operatorname{Tr}(H), \ \text{ and}\  \operatorname{tr}(\theta^{k})=\operatorname{Tr}(H^{k}), \quad, 1 \leq k \leq n-1.
$$
Let $\alpha=\alpha_{0}+\alpha_{1} \theta+\cdots+\alpha_{n-1} \theta^{n-1} \in E$, it follows that
$$
\begin{aligned}
\operatorname{tr}(\alpha) =\sum_{k=0}^{n-1} \alpha_{i} \operatorname{tr}\left(\theta^{k}\right)=\sum_{i=0}^{n-1} \alpha_{i}
\operatorname{Tr}\left(H^{k}\right) =\operatorname{Tr}\left(\sum_{k=0}^{n-1} \alpha_{i } H^{k}\right)=\operatorname{Tr}\left(H^{*}(\overline{\alpha})\right) .
\end{aligned}
$$

To show that conclusion on the norm, let $\alpha^{(i)}(0 \leq i \leq n-1)$ be the $n$ conjugations of $\alpha$ in the
 smallest normal extension of $Q$ containing $E$, where $\alpha^{(0)}=\alpha=\alpha_{0}+\alpha_{1} \theta+\cdots+ \alpha_{n-1} \theta^{n-1}$.
  It is easily seen that
$$
\alpha^{(i)}=\sum_{k=0}^{n-1} \alpha_{k} \theta_{i}^{k},\  \text { where }\  \theta_{0}=\theta \ \text { and } \ 0 \leq i \leq n-1.
$$
By property (iii) of lemma \ref{lm2.4}, we have
$$
N(\alpha)=\prod_{i=0}^{n-1} \alpha^{(i)}=\prod_{i=0}^{n-1} \alpha\left(\theta_{i}\right)=\operatorname{det}\left(H^{*}(\overline{\alpha})\right).
$$
We complete the proof of Theorem  \ref{th2}.
\end{proof}

The cyclic lattice in $\mathbb{R}^{n}$ was introduced by Micciancio in \cite{M3}, (also see \cite{Z}), which  plays an important role in Ajtai's
 construction  of collision resistant  Hash function( see \cite{A}). As an application, we show that every  ideal in an algebraic
  number field corresponds to a cyclic lattice:

  \begin{cor}
Let $ A\subset R $ be an ideal  and $A\neq 0$, then $\tau(A) \subset Q^{n}$ is a cyclic lattice.
\end{cor}
\begin{proof}
Suppose that $\alpha \in A$. Since $\theta \in R$, then $\theta \alpha \in A$. By  (\ref{2.16}), we have
$$
\tau(\theta \alpha)=H \overline{\alpha} \in \tau(A) .
$$
Thus $\tau(A)$ is a cyclic lattice.
\end{proof}

\section{ High Dimensional RSA}

In this section, we give an attainable algorithm for the high dimensional RSA by making use of lattice theory, this algorithm is significant both from the theoretical and practical point of view. Suppose that the algebraic numbers field $E$ satisfying the NC-property, then $R=\mathbb{Z}[\theta]$ is the ring of algebraic integers of $E$, the restriction of correspondence $\tau$ gives a ring isomorphism from $R$
to $\mathbb{Z}^{n}$. Let $\mathbb{Z}(x)$ be the ring of integer coefficients polynomials and $(\phi(x))$ be the principal ideal generated by $\phi(x)$ in $\mathbb{Z}(x)$, it is easy to see that $R \cong \mathbb{Z}[x] / (\phi(x))$. Let $M_{\mathbb{Z}}^{*}$ be the set of ideal matrices generated by an integral vector, i.e.
\begin{equation}\label{3.1}
M_{\mathbb{Z}}^{*}=\left\{H^{*}(\overline{f}) \mid \overline{f} \in \mathbb{Z}^{n}\right\} .
\end{equation}
Then the following four rings are isomorphic from each other
\begin{equation}\label{3.2}
\mathbb{Z}[x]/(\phi(x)) \cong R\cong \mathbb{Z}^{n} \cong M_{\mathbb{Z}}^{*}.
\end{equation}

For any polynomial $ \alpha(x)=\alpha_{0}+\alpha_{1}x+\cdots+\alpha_{n-1}x^{n-1} \in \mathbb{Z}[x]/(\phi(x))$, the corresponding algebraic integer is
$ \alpha=\alpha_{0}+\alpha_{1}\theta+\cdots+\alpha_{n-1}\theta^{n-1}\in R$, we write this  isomorphism by
\begin{equation}\label{3.3}
\alpha(x) \rightarrow \alpha \stackrel{\tau}{\longrightarrow} \overline{\alpha} \stackrel{H^{*}}{\longrightarrow} H^{*}(\alpha).
\end{equation}

A $ \phi$-ideal lattice means an integer lattice of which corresponds an ideal of $\mathbb{Z}(x) /(\phi(x))$, it was first introduced by Lyubashevsky and  Micciancio in $ $(see also \cite{Z}), which also plays a key role in  Gentry's construction for the full homomorphic cryptosystem (see \cite{G-a}), Fluckiger and Suarez in \cite{F} extended this definition to total real number field. .

\begin{lem}\label{lm3.1}  Let $E$  be  an  algebraic numbers field with  NC- property, $R=\mathbb{Z}[\theta]$ be the ring of algebraic integers of $E$. Then there is an one to one correspondence  between ideals of $R$ and the $\phi$-ideal  lattices. Moreover, if $\alpha \in R$, then we have
\begin{equation}\label{3.4}
\tau(\alpha R)=L\left(H^{*}(\overline{\alpha})\right).
\end{equation}
In general, suppose that $A \subset R$ is an ideal and $A \neq 0$, then there exists two elements $\alpha$ and $\beta$ in $A$ such that
\begin{equation}\label{3.5}
\tau(A)=L\left(H^{*}(\overline{\alpha})\right)+L\left(H^{*}(\overline{\beta})\right).
\end{equation}
 \end{lem}
 \begin{proof}
 Since there is an one to one correspondence between the $\phi$-ideal lattices and the ideals of $\mathbb{Z}[x] /(\phi(x))$ (See Corollary  of \cite{Z}), by (\ref{3.2}), the first assertion follows immediately. Let $\alpha \in R$, then $\alpha R=\{\alpha x \mid x \in R\}$, by lemma \ref{lm2.5} we have
$$
\tau(\alpha x)=H^{*}(\alpha) \overline{x},\ \  \text { where } \overline{x}= \begin{pmatrix}
     x_0 \\
    x_1  \\
     \vdots \\
     x_{n-1}
\end{pmatrix}\in \mathbb{Z}^{n} .
$$
It follows what
$$
\tau(\alpha R)=\left\{H^{*}(\alpha) \overline{x} \mid \overline{x}\in \mathbb{Z}^{n}\right\}=L\left(H^{*}(\overline{\alpha})\right).
$$
To prove (\ref{3.5}), it is known  that any an ideal of $R$ is generated by at most two elements (see corollary 5 of page 11 of \cite{N} ), namely, $A=\alpha R+\beta R$, then we have
$$
\tau(A)=\tau(\alpha R)+\tau(\beta R)=L\left(H^{*}(\overline{\alpha})\right)+L\left(H^{*}(\overline{\beta})\right).
$$
\end{proof}

To introduce an attainable algorithm for high dimensional RSA, we require some basic results from lattice theory. Let $L=L(B)\subset \mathbb{R}^{n}$ be a full-rank lattice, the determinant of $L$ is defined by
\begin{equation}\label{3.6}
d(L)=|\operatorname{det}(B)|.
\end{equation}

Suppose that the generated matrix $B=\left[\overline{b}_{1}, \overline{b}_{2}, \cdots, \overline{b}_{n}\right], \overline{b}_{i} \in \mathbb{R}^{n}$ is the column vectors of $B$. Since $\left\{\overline{b}_{1}, \overline{b}_{2}, \cdots, \overline{{b}}_{n}\right\}$ is a basis for $\mathbb{R}^{n}$, let $B^{*}=\left\{\overline{b}_{1}^{*}, \overline{b}_{2}^{*}, \cdots, \overline{b}_{n}^{*}\right\}$ be the corresponding orthogonal basis, where $\overline{b}_{1}^{*}=\overline{b}_{1}$, and  $\overline{b}_{i}^{*} $ is obtained by Gram-Schmidt orthogonal process in order.

A basis $B$ is called in Hermited Normal Form (HNF) if it is upper triangular, all elements on the diagonal are strictly positive, and any other elements $b_{i j}$ satisfies $0 \leq b_{i j}<b_{i i}$. It is easy to see that every integer lattice $L=L(B)$ has a unique basis in Hermited Normal Form,  denoted by $\operatorname{HNF}(L)$(see Theorem 2.4.3 of \cite{C2}). Moreover, given any basis $B$ for lattice $L,$ $ \operatorname{HNF}(L)$ can be efficiently computed from $B$ (see \cite{M4} and \cite{C2}).

\begin{prop}\label{pp1}
Let $L=L(B)$  and $B=(b_{ij})_{n\times n}$ be the basis in HNF. Then the corresponding orthogonal basis $B^{*}$ is a diagonal matrix, namely
\begin{equation}\label{3.7}
B^{*}=\operatorname{diag}\left\{b_{11}, b_{22}, \cdots, b_{nn}\right\}.
\end{equation}
Moreover, we have
\begin{equation}\label{3.8}
d(L)=\prod_{i=1}^{n} b_{i i}.
\end{equation}
\end{prop}
\begin{proof}
See \cite{M4}.
\end{proof}

\begin{def}\label{df3.1}
Let $L=L(B)\subset \mathbb{R}^{n}$ be a full-rank lattice, and $B^{*}=\left[\overline{b}_{1}^{*}, \overline{b}_{2}^{*}, \cdots, \overline{b}_{n}^{*}\right]$ be the corresponding  orthogonal basis, the orthogonal  parallelepiped $F\left(B^{*}\right)$ is defined by
\begin{equation}\label{3.9}
F( B^{*})=\left\{\sum_{i=1}^{n} x_{ i}\overline{b}_{i}^{*} \mid 0\leq x_{i}<1 \ \text{and}\ x_{i}\in \mathbb{R}\right\}.
\end{equation}
\end{def}

\begin{prop}\label{pp2}
Let $L=L(B)\subset \mathbb{Z}^{n}$ be an integer lattice,  $B= \operatorname{HNF}(L)$ be  the basis in $\operatorname{HNF}$ and $B^{*}= \operatorname{diag}\left\{b_{11}, b_{22}, \cdots, b_{nn}\right\}$ be the corresponding orthogonal  basis, $F\left(B^{*}\right)$ is the
orthogonal  parallelepiped given by (\ref{3.9}), then $S$ is a set of coset representatives for the quotient group $\mathbb{Z}^{n} / L$, where
$$
S=F\left(B^{*}\right) \cap \mathbb{Z}^{n}=\left\{x^{\prime}=\left(x_{1}, x_{2}, \cdots, x_{n}\right) \mid \forall x_{i} \in \mathbb{Z}\ \
\text { and } \ \ 0 \leq x_{1}<b_{ ii}\right\}.
$$
\end{prop}

\begin{proof}
See section  4.1 of \cite{M4}.
\end{proof}

Now, we return to the algebraic numbers field $E=Q[\theta]$ (with NC-property). Let $\alpha,\beta \in R$ be two algebraic integers, by Lemma \ref{lm3.1}, the principal ideal $\alpha R$ corresponds to the minimal  $ \phi$-ideal lattice $L ( H^{*}(\overline{\alpha}))$.  Thus $A=(\alpha R)(\beta R)=\alpha \beta R$ corresponds  to $L\left( H^{*}( \overline{\alpha}\otimes \beta)\right)$.

\begin{defn}\label{df3.2}

For given $\alpha, \beta \in R$, $ \tau(\alpha)=\overline{\alpha}$ and $\tau(\beta)=\overline{\beta}$, we denote the lattice $L_{\alpha, \beta}$ by
\begin{equation}\label{3.11}
L_{\alpha, \beta}=L\left(H^{*}(\overline{\alpha} \otimes \overline{\beta})\right).
\end{equation}
The $\operatorname{HNF}$ basis of $L_{\alpha, \beta}$ is denoted by $B_{ \alpha, \beta}$ and the corresponding orthogonal basis is denoted by
\begin{equation}\label{3.12}
B_{\alpha, \beta}^{*}=\operatorname{diag}\left\{b_{1}, b_{2}, \cdots, b_{n}\right\},
\end{equation}
where  $b_{i} \in \mathbb{Z}$  and $b_{i} \geq 1$. The parallelepiped is given by
\begin{equation}\label{3.13}
S_{\alpha, \beta}=\left\{\left(x_{1}, x_{2}, \cdots, x_{n}\right) \in \mathbb{Z}^{n} \mid x_{i} \in \mathbb{Z}\ \  \text { and } \ \ 0 \leq x_{i}<b_{i}\right\}.
\end{equation}
\end{defn}

\begin{lem}\label{lm3.2}
 Let $\alpha \in R, \beta \in R$ and $A=\alpha \beta R$. Then $S_{\alpha, \beta} $ given by (\ref{3.13}) is corresponding to a set of coset representatives of the factor ring $R/A$ in the algebraic numbers field $E$  with  NC-property.
\end{lem}

\begin{proof}
 By Proposition \ref{pp1}, it is  easy  to see that
$$
\left|S_{\alpha, \beta}\right|=\prod_{i=1}^{n} b_{i} =\left|\operatorname{det}\left(H^{*}(\overline{\alpha} \otimes \overline{\beta})\right)\right| =\left| \operatorname{det}\left(H^{*}(\overline{\alpha})\right) \right|  \cdot \left| \operatorname{det}\left(H^{*}(\overline{\beta})\right) \right|=d\left(L_{\alpha, \beta}\right).
$$
By theorem \ref{th2} and $(\ref{1.12})$, we have
$$
\begin{aligned}
N(A) =|N(\alpha  \cdot  \beta)|=| N( \alpha)| \cdot | N(\beta) |= \left| \operatorname{det}\left(H^{*}(\overline{\alpha})\right)\right|\cdot\left| \operatorname{det}\left(H^{*}(\overline{\beta})\right)\right|
=d\left(L_{\alpha, \beta}\right).
\end{aligned}
$$
It follows that $N(A)=\left|S_{\alpha, \beta}\right|$. Since $E$ satisfies NC-property, if $\alpha \in R$, then $\overline{\alpha}=\tau(\alpha) \in \mathbb{Z}^{n}$, hence $\alpha \equiv \beta(\text{mod}\ \ A)$ in $R$, if and only  if
$$
\overline{\alpha} \equiv \overline{\beta}\left(\operatorname{mod}\ \  L_{\alpha, \beta}\right) .
$$
The lemma follows from Proposition \ref{pp2} immediately.
 \end{proof}

The main result of this subsection is the following theorem.

\begin{thm} \label{th3}  Let $E$ be an algebraic numbers field of degree  $n$ with NC-property, $\alpha \in R, \beta \in R$ be two distinct prime elements, $A=\alpha \beta R$, and $L_{\alpha, \beta}$ be the lattice given  by (\ref{3.11}). Then for any $\overline{a} \in \mathbb{Z}^{n}, k \in \mathbb{Z}, k \geq 0$, we have
\begin{equation}\label{3.14}
\overline{a}^{\otimes(k \varphi(\alpha, \beta)+1)}\equiv \overline{a} \left(\text { mod } \ \ L_{\alpha, \beta}\right),
\end{equation}
where
\begin{equation}\label{3.15}
\varphi(\alpha, \beta)=\left( \left|\operatorname{det}\left(H^{*}(\overline{\alpha})\right)\right|-1\right) \left( \left|\operatorname{det}\left(H^{*}(\overline{\beta})\right)\right|-1\right) .
\end{equation}
\end{thm}

\begin{proof}
 Since $E$ satisfies   NC-property, $\overline{a} \in \mathbb{Z}^{n}$, then $a=\tau^{-1}(\overline{a}) \in R$. By Theorem \ref{th1} , we have
$$
a^{k \varphi(A)+1} \equiv a ( \text { mod }\ \  A ).
$$
It is easy to see that
$$
\begin{aligned}
\varphi(A)=\varphi(\alpha R) \varphi(\beta A)&=(N(\alpha R)-1)(N(\beta R)-1) \\
&=(|N(\alpha)|-1)(|N(\beta)|-1) \\
&=\left(\left|\operatorname{det}\left(H^{*}(\overline{\alpha})\right)\right|-1\right)\left(\left|\operatorname{det}\left(H^{*}(\overline{\beta})\right)\right|-1\right) \\
&=\varphi(\alpha, \beta).
\end{aligned}
$$
By lemma \ref{lm3.1}, we have
$$
\tau(A)=\tau(\alpha \beta R)= L\left(H^{*}(\overline{\alpha}\otimes \overline{\beta})\right)=L_{\alpha, \beta}
\ \ \text{and}\ \
\tau\left(a^{k \varphi(\alpha,\beta)+1}\right)=\overline{a}^{\otimes(k\varphi(\alpha, \beta)+1)}.
$$
Therefore, (\ref{3.14}) follows immediately.
\end{proof}

According to the above theorem, we may describe an attainable algorithm for high dimensional RSA as follows.

  \begin{table}[h]
\begin{tabular}{l}
\hline
  {\bf{\text{ Algorithm I: \ \ RSA in the Algebraic Numbers field}}}\\
\hline
\\
\ \ \ $n \geq 1$ is a positive integer, $E / Q$ is an algebraic numbers field with NC-property of \\
\ \ \ degree $n$,
 $R \subset E$ is the ring of algebraic integers of $E$, $\alpha \in R$, $ \beta \in R $
 are two distinct  \\  \ \ \ prime elements of $R$,
  $A=\alpha \beta R$ is a principal ideal of $R,$
 $H^{*}(\overline{\alpha} \otimes \overline{\beta})$ is the  ideal  \\ \ \ \ matrix corresponding to $A$,
 $L_{\alpha, \beta}=L\left(H^{*}(\overline{\alpha}\otimes \overline{\beta})\right)$ is the lattice generated by \\ \ \ \ $H^{*}(\overline{\alpha}\otimes \overline{\beta})$,
   $B_{\alpha, \beta}=\text{HNF}\left(L_{\alpha, \beta}\right)$ is the basis of $L_{\alpha,\beta}$ in HNF, \\
  \ \ \  $B_{\alpha, \beta}^{*}=\operatorname{diag}\left\{b_{1}, b_{2}, \cdots, b_{n}\right\}$ is the  corresponding orthogonal basis.\\

$\bullet$\ \ {\bf{Parameters: }}\ \ $ \varphi( \alpha, \beta)=\left(\left|\operatorname{det}\left(H^{*}(\overline{\alpha})\right)\right|-1\right)
\left(\left|\operatorname{det}\left(H^{*}(\overline{\beta})\right)\right|-1\right)$,\\
\ \ \ \ \ \ \ \ \ \ \ \ \ \ \ \ \ \ \ \ \ \ $S_{\alpha, \beta}=\{ x'=(x_{1},x_{2},\cdots, x_{n})\in \mathbb{Z}^{n} \mid 0\leq x_{i}<b_{i}\}.$
$ 1\leq e <\varphi (\alpha,\beta)$, \\
\ \ \ \ \ \ \ \ \ \ \ \ \ \ \ \ \ \ \ \ \ \  $ 1\leq d <\varphi (\alpha,\beta),$
such that $ed\equiv 1(\text{mod}\ \ \varphi (\alpha,\beta))$.\\

\noindent$\bullet$\ \ {\bf{Public keys:}}\ \  The rotation matrix $H$, the lattice $L( B_{\alpha, \beta})=L_{\alpha, \beta}$
and the \\
\ \ \ \ \ \ \ \ \ \ \ \ \ \ \ \ \ \ \ \ \ \ positive integer $e$ are public keys. \\

 \noindent$\bullet$\ \ {\bf{Private keys:}}\ \  Ideal matrices $H^{*}(\overline{\alpha}), H^{*}(\overline{\beta})$,  the basis $H^{*}(\overline{\alpha} \otimes \overline{\beta})$ of $L_{\alpha, \beta}$ \\ \ \ \ \ \ \ \ \ \ \ \ \ \ \ \ \ \ \ \ \ \ \ and positive integer $d$ are private keys.\\

 \noindent$\bullet$\ \ {\bf{Encryption:}} \ \  For any input message $\overline{a} \in S_{\alpha , \beta}$, the ciphertext
$\overline{c} $  is given  by\\ \ \ \ \ \ \ \ \ \ \ \ \ \ \ \ \ \ \ \ \ \ \
$
\overline{c} \equiv \overline{a} ^{ \otimes e} (\operatorname{mod}\ \  L_{\alpha, \beta}).
$\\
 \noindent$\bullet$\ \ {\bf{Decryption:}} \ \
 $\overline{c}^{\otimes d} \equiv \overline{a}^{\otimes  d e} \equiv \overline{a}^{\otimes( k\varphi(\alpha, \beta)+1)} \equiv \overline{a}(\text{mod}\  L_{\alpha, \beta}).$ One can find the plaintext \\
 \ \ \ \ \ \ \ \ \ \ \ \ \ \ \ \ \ \ \ \ \ \ $ \overline{a}$ from $\overline{c}$ in $S_{\alpha, \beta}$.\\
\hline
\end{tabular}
\caption{ Algorithm I}
\end{table}

\newpage
\begin{rem}
If the class number $h_{E}=1$, in other words, $R$ is a UFD, then the prime elements is equivalent to irreducible elements in $R$,
and one can find prime elements  $\alpha$ from $\alpha(x) \in \mathbb{Z}[x]/(\phi(x))$ and $\alpha(x)$  irreducible.

\end{rem}

\section{ Security and Example}
The classical RSA public key cryptosystem is nowadays   used in a wide variety of applications ranging from web browsers to smart cords. Since its initial publication in 1978, many researchers have tried to look for vulnerabilities in the system. Some clever attacks have been found (see \cite{B} and \cite{C} ). However, none of the  known attacks is devastating and the ordinary RSA system is still considered secure.

The security of high dimensional RSA depends on virtually factoring of an element of the algebraic integers ring $R$ into product of of
distinct  prime elements. Factoring on $R$ is much more complicate  than factoring of a positive integer, none of efficient  method is known
  up to day, thus we consider the high dimensional RSA  almost absolutely  secure.

To see the size of private keys, since $\operatorname{det}\left(H^{*}(\overline{\alpha})\right)=N(\alpha)$, it may be extremely huge, for example, if $\alpha=p \in \mathbb{Z},$ $ \beta=q \in \mathbb{Z}$ are prime numbers, then
$$
\operatorname{det}\left(H^{*}(\overline{\alpha})\right)=N(\alpha)=p^{n},\ \ \  \operatorname{det}\left(H^{*}(\overline{\beta})\right)=q^{n}
$$
and
$$\varphi(\alpha, \beta)=\left(p^{n}-1\right)\left(q^{n}-1\right),$$
which is much larger than $p q$, the later is the site of public key  of the classical  RSA cryptosystem.

The lattice based on cryptography have been intensively studied for past two decades. The GGH cryptosystem proposed by Goldreich, Goldwasser and  Halevi in \cite{G}, which is perhaps the most intuitive encryption scheme based on lattices. The public key is a "bad" basis for a lattice,  Micciancio proposed in \cite{M4} to use, as the public basis, the Hermite Normal Form $B=$HNF$(L)$. The private key of GGH is  an  exceptionally good basis  for $L$. The security of GGH  relies on the assumption that it is difficult to find a spacial basis for $L$ from a known basis of  $L$. In this sense, we regard the high dimensional RSA as secure as GGH/HNF cryptosystem at least.

Another number theoretic cryptosystem based on lattice is NTRUEncrypt. The public key cryptosystem NTRU proposed in 1996 by Hoffstein, Pipher and Silverman in \cite{H}, is the fastest known lattice based encryption scheme, although its description relies on arithmetic over polynomial quotient ring $Z[x]/\langle x^{n}-1\rangle$, it was easily observed that it could be expressed as a lattice based on cryptosystem.
NTRU uses a q-ary convolutional modular lattice(see \cite{M2} and \cite{Z3}), its public key is also the HNF basis of L and the private key is 
a special basis of L containing two secrete polynomials $f(x)$ and $g(x)$. Obviously, our algorithm I is at least as hard as solving NTRUEncrypt.

Unfortunately, neither GGH nor NTRU is supported by a proof of security showing that breaking the cryptosystem is at least as hard as solving
some underlying lattice problem; they are primarily  practical proposals aimed at offering a concrete alternative to RSA or other number theoretic cryptosystems(see page 166 of \cite{M2}). However, the significance of this paper is to show that the real alternative of RSA   
is the high dimensional RSA we present here rather than GGH and NTRU.
\begin{example}
 Finally, we give an example and see how to work of the high dimensional RSA in a quadratic field. Let $E=Q(\sqrt{d})$,  $d \in \mathbb{Z}$ be a square-free integer and $d\equiv 2,$  or $3 \  \mathrm{mod} \ 4$, thus  $E$ satisfies the NC-property. Let $\delta_{E}$ be the discriminant of $E$, it is known that $\delta_{E}=4 d$( see Proposition 13.1.2 of \cite{I}). Let  $p \in \mathbb{Z}$ be an odd prime satisfying  the following condition
 \begin{equation}\label{4.1}
p \nmid 4d,\ \ \text{ and }\ \ x^{2}\equiv d(\text{mod}\ \  p) \  \text{is not solvable in }\ \mathbb{Z}.
\end{equation}
By Proposition 13.1.3 of \cite{I}, we know that $p$ is a prime element in $E$.

According to Algorithm $\mathrm{I}$, we select two large primes $p$ and $q$ of which satisfying (\ref{4.1}). Let $\alpha=p$ and $\beta=q$, then
$$
\bar{\alpha}=\left(\begin{array}{l}
p \\
0
\end{array}\right),\  \bar{\beta}=\left(\begin{array}{l}
q \\
0
\end{array}\right),  \  H^{*}(\overline{\alpha})=\left(\begin{array}{ll}
p & 0 \\
0 & p
\end{array}\right),\  \text { and }\  H^{*}(\overline{\beta})=\left(\begin{array}{ll}
q & 0 \\
0 & q
\end{array}\right).
$$
It follows that
\begin{equation}\label{4.2}
H^{*}( \overline{\alpha} \otimes \overline{\beta})= H^{*}( \overline{\alpha} )H^{*}( \overline{\beta} )=\left(\begin{matrix}
pq&0\\
0&pq
\end{matrix}\right),\ \ L_{\alpha,\beta}=L\left(H^{*}( \overline{\alpha} \otimes \overline{\beta}) \right)
\end{equation}
and
\begin{equation}\label{4.3}
S_{\alpha, \beta}=\left\{ x=\left(\begin{matrix}
x_{1}\\
x_{2}
\end{matrix}\right) \in \mathbb{Z}^{2}\mid 0\leq x_{1}, x_{2} <pq \right\}.
\end{equation}
It is easy to see that
\begin{equation}\label{4.4}
\varphi(\alpha,\beta)=(p^{2}-1)(q^{2}-1).
\end{equation}
\end{example}

In this  special case, the 2-dimensional RSA maybe described as follows.

 \begin{table}[h]
\begin{tabular}{l}
\hline
  {\bf{ \ \ RSA in a  Quadratic Field  }}\\
\hline
\\
$\bullet$\ \ {\bf{Parameters:}} \ \ $ E=Q(\sqrt{d})$, $d$ is a square-free integer and $d\equiv 2$ or $3(\text{mod}\ \ 4)$,\\
\ \ \ \ \ \ \ \ \ \ \ \ \ \ \ \ \ \ \ \ \ \ the rotation matrix $H= \left(\begin{matrix}
0 & d\\
1& 0
\end{matrix}\right) $, $p,q$ are two large and distinct  \\
\ \ \ \ \ \ \ \ \ \ \ \ \ \ \ \ \ \ \ \ \ \ prime numbers of which satisfy (\ref{4.1}).  $ N=p  q$ and $\chi(N)=\left(p^{2}-1\right)\left(q^{2}-1\right).$\\
\ \ \ \ \ \ \ \ \ \ \ \ \ \ \ \ \ \ \ \ \ \
$L=L(B)$ is a lattice, $B=\left(\begin{array}{cc}\mathrm{N} & 0 \\ 0 & \mathrm{N}\end{array}\right) .$ $1 \leq e<\chi(N),$ $ 1 \leq d_{1} <\chi(N)$\\
\ \ \ \ \ \ \ \ \ \ \ \ \ \ \ \ \ \ \ \ \ \   such that $e d_{1} \equiv 1(\operatorname{mod}\ \  \chi(N))$.\\

\noindent$\bullet$\ \ {\bf{Public keys:}}\ \  $H, N$ and the positive integer $e$ are public keys.\\

 \noindent$\bullet$\ \ {\bf{Private keys:}}\ \  $p$, $q$ and the positive integer $d_{1}$ are private keys.\\

 \noindent$\bullet$\ \ {\bf{Encryption:}} \ \  For any $a=\left(\begin{array}{l}a_{1} \\ a_{2}\end{array}\right) \in \mathbb{Z}_{pq}^{2}$, the ciphertext $c=\left(\begin{array}{l} c_{1} \\ c_{2}\end{array}\right) \in \mathbb{Z}^{2}$ \\ \ \ \ \ \ \ \ \ \ \ \ \ \ \ \ \ \ \ \ \ \ \
given by $
c\equiv a^{ \otimes e} (\operatorname{mod}\ \  L).$\\
 \noindent$\bullet$\ \ {\bf{Decryption: }}\ \
 $c^{\otimes d_{1}} \equiv a^{\otimes  d_{1}e} \equiv a(\text{mod}\  L).$ One can find the plaintext $ a$ from $c$ in $\mathbb{Z}_{pq}^{2}$.\\
\hline
\end{tabular}
\caption{RSA in a  Quadratic Field}
\end{table}

We can similarly deal with the cases of Cyclotomic Fields. Let $n=\varphi(m)$ for some positive integers $m$, $\xi_{m}=e^{2\pi i /m},$
$E=Q(\xi_{m})$  and $R \subset E$ be the ring of algebraic integers of $E.$ Suppose that $p\in \mathbb{Z}$  is a rational prime number, then $p$ is a prime element of $R$ if and only if (see Theorem 2 of page 196 of \cite{I})
\begin{equation}\label{4.2}
p \nmid m \ \ \text{and}\ \ p^{\varphi(m)} \equiv 1(\text{mod}\ \ m).
\end{equation}
Suppose that $p \in \mathbb{Z}$ and $q\in \mathbb{Z}$ are two distinct prime numbers satisfying (\ref{4.2}), we obtain  the lattice $L(H^{*}(\overline{p} \otimes \overline{q} ))$ and an attainable algorithm in $Q(\xi_{m}). $
 \vspace{-0.4cm}{\footnotesize}

\end{document}